\documentclass{article}
\usepackage[utf8]{inputenc}
\usepackage{amsmath,amsfonts,amssymb,amsthm}
\usepackage{mdframed}
\usepackage{xcolor}
\usepackage[round]{natbib}
\usepackage{graphicx}
\usepackage[normalem]{ulem}
\allowdisplaybreaks

\renewcommand{\emptyset}{\varnothing}
\renewcommand{\le}{\leqslant}
\renewcommand{\leq}{\leqslant}
\renewcommand{\ge}{\geqslant}
\renewcommand{\geq}{\geqslant}

\newcommand{\rd}{\,\mathrm{d}}
\newcommand{\mrd}{\mathrm{d}}
\newcommand{\phe}{\phantom{=}}
\newcommand{\tod}{\stackrel{\mathrm{d}\, }\to}

\newcommand{\real}{\mathbb{R}}
\newcommand{\nat}{\mathbb{N}}
\newcommand{\e}{\mathbb{E}}
\newcommand{\var}{\mathrm{Var}}
\newcommand{\cov}{\mathrm{Cov}}
\newcommand{\giv}{\!\mid\!}

\newcommand{\dnorm}{\mathcal{N}}

\newcommand{\bsc}{\boldsymbol{c}}
\newcommand{\bsx}{\boldsymbol{x}}
\newcommand{\bsy}{\boldsymbol{y}}
\newcommand{\bsz}{\boldsymbol{z}}
\newcommand{\bsone}{\boldsymbol{1}}

\newcommand{\umu}{\underline{\mu}}
\newcommand{\omu}{\overline{\mu}}

\newcommand{\ult}{\underline{\tau}}
\newcommand{\olt}{\overline{\tau}}

\newcommand{\osigsq}{{\overline\sigma}^2}

\newcommand{\tran}{\mathsf{T}}

\newtheorem{theorem}{Theorem}
\newtheorem{lemma}{Lemma}
\newtheorem{proposition}{Proposition}
\newtheorem{corollary}{Corollary}
\theoremstyle{definition}
\newtheorem{remark}{Remark}

\title{Mean dimension of radial basis functions}
\author{Christopher Hoyt\\Stanford University
\and 
Art B. Owen\\ Stanford University
}
\date{February 2023}

\begin{document}
\maketitle
\begin{abstract}
We show that generalized
multiquadric radial basis functions (RBFs) on $\real^d$ have a mean dimension that is $1+O(1/d)$
as $d\to\infty$ with an explicit bound for the implied constant, under moment conditions on
their inputs.  Under weaker moment conditions
the mean dimension still approaches $1$. As a consequence, these RBFs become 
essentially additive as their dimension increases.  
Gaussian RBFs by contrast can attain any
mean dimension between 1 and d.  We also find that  a test integrand due to Keister that has been influential in quasi-Monte Carlo theory has a mean dimension that oscillates between approximately 1 and approximately 2
as the nominal dimension $d$ increases.
\end{abstract}

\section{Introduction}

For high dimensional functions it is very useful to find
parameterizations in terms of some vectors
of the same dimension as the input space.
Two such parameterizations are ridge functions
$\phi(\bsx^\tran\theta)$ and radial basis functions (RBFs)
$\phi( \Vert\bsx-\bsc\Vert)$ for vectors $\bsc,\theta$
of the same dimension as $\bsx$ and appropriate functions
$\phi(\cdot)$.
In this paper we study RBFs.  We are interested in them
because of their connections to multiple numerical
problems of interest in statistics and other
mathematical sciences: 
interpolation, machine learning,
Gaussian process regression (kriging),  
and multidimensional integration.

There are some results based on concentration of measure wherein high dimensional Lipschitz functions become essentially constant as the dimension $d$ of their domain tends to infinity.  See for example \cite{dono:2000}.
In this paper we study the way in which some 
of these high
dimensional functions $\phi(\cdot)$ fluctuate around their 
nearly constant value.  Our main results are that certain RBFs must become essentially additive
as $d\to\infty$, while others are not so constrained.  Our techniques are based on
the functional ANOVA decomposition of
\cite{hoef:1948}, \cite{sobo:1969} and \cite{efro:stei:1981}. A function of $d$ independent
variables has $2^d-1$ non-trivial 
variance components $\sigma^2_u$ for nonempty
subsets $u$ of variables.  The mean dimension
is the weighted average of cardinalities $|u|$
with weights proportional to $\sigma^2_u$.  It can take 
values between $1$ and $d$.
A mean dimension near
one means that the function is nearly additive in a
least squares sense.

We find that some classic RBFs such as multiquadrics
\citep{hard:1971} have 
mean dimension $1+O(1/d)$ as the
dimension $d\to\infty$. If a function has mean dimension
$1+\epsilon$ then it has an additive approximation that
explains at least $1-\epsilon$ of its variance.
The well-known 
Gaussian RBF (that we define below) need not be of
low mean dimension.  We show that it can be
parameterized to attain any mean dimension in
the interval $(1,d)$ when the variables in it
have continuous distributions. This RBF is known
as the Gaussian RBF in machine learning, but
some other literatures call it the
squared exponential RBF.

To fix ideas, suppose that we have measured values
$f(\bsx_i)$ for $\bsx_i\in\real^d$ and $i=1,\dots,n$.
We seek an interpolant $\tilde f(\bsx)$ for
$\bsx\in\real^d$.  We might then use
$$
\tilde f(\bsx) = \sum_{i=1}^n\beta_i\phi(\Vert\bsx-\bsx_i\Vert)
$$
after solving $n$ equations in $n$ unknowns
to compute $\beta=(\beta_1,\dots,\beta_n)^\tran\in\real^n$.
Only certain special functions $\phi$ are
good choices for this usage.
We describe some of those in a later section
based on material from \cite{fass:2007}.
For now we mention generalized multiquadrics
$\phi(\Vert\bsx\Vert)=(a+\Vert\bsx\Vert\vartheta^2)^p$ 
and Gaussians, $\phi(\Vert\bsx\Vert)=\exp(-\Vert\bsx\Vert^2\vartheta^2)$
for parameters $p\in\real\setminus\nat_0$ and $a\ge0$
and $\vartheta>0$.

Now suppose that $f(\bsx)$ is not nearly additive but
all of the $\phi(\Vert\bsx-\bsx_i\Vert)$ are nearly additive.
It would still be possible to interpolate if the
$\beta_i$ took values of large magnitude with
opposite signs that mostly
cancelled the additive parts in $\phi(\cdot)$.
We would however expect serious numerical
conditioning difficulties in that setting.
RBF approximation is often ill-conditioned
even with functions that are not nearly additive.
Fitting a non-additive function by nearly additive
basis functions can only make things worse.

The covariance functions used in Gaussian process regression
often take the RBF form, especially in geoscience.
An additive covariance function implies additive
realizations of the random field, a potentially
serious limitation.  This may be why 
covariances of product form are more popular than covariance
models of the RBF form in high dimensional Gaussian
process models such as those used in computer
experiments \citep{sack:welc:mitc:wynn:1989}.

An important test function for quasi-Monte Carlo (QMC)
integration is the Keister function from \cite{keis:1996}.
This is a radial basis function.  Although it is expressed
as a sinusoidal function of $\Vert\bsx\Vert^2$ for Gaussian
$\bsx$, making all $d$ variable equally
important, we will see that it is generally of low mean dimension.

The asymptotic mean dimension of ridge
functions was studied in \cite{hoyt:owen:2020} for 
$\bsx\sim\dnorm(0,I)$.  If the ridge
function $\phi(\cdot)$ is Lipshitz
continuous, then the mean dimension
of $f(\bsx)=\phi(\bsx^\tran\theta)$
for a unit vector $\theta\in\real^d$
remains bounded as the nominal dimension $d\to\infty$.
Some discontinuous ridge functions can
have mean dimensions that grow proportionally to $\sqrt{d}$.
A form of conditional QMC known as
pre-integration (see \cite{grie:kuo:leov:sloa:2018})
can convert them into Lipshitz
continuous ridge functions, greatly
reducing their asymptotic mean dimension, which then makes them easier
to integrate numerically

An outline of this paper is as follows.
Section~\ref{sec:notation} introduces our notation,
gives some properties of RBFs, and presents the
functional ANOVA and related material for mean dimension.
Section~\ref{sec:meandimmultiq} shows that, under some moment conditions, generalized
multiquadric RBFs have mean dimension $1+O(1/d)$ as $d\to\infty$ with an
explicit upper bound on the implied constant
in the $O(1/d)$ term.
Under much weaker moment conditions, the mean dimension still
approaches $1$ as $d\to\infty$.
Section~\ref{sec:gaussian} shows that the Gaussian RBFs
can attain any mean dimension in the interval $(1,d)$ when the inputs have continuous
distributions with bounded densities having
support near $\bsc$.
Section~\ref{sec:keister} shows that the
mean dimension of the Keister function
oscillates between nearly $1$ and nearly $2$ as the nominal dimension $d$ increases.
Section~\ref{sec:conclusions}
discusses how mean dimension varies among alternative
methods. 
Finally, there are appendices for the
lengthier proofs.

\section{Notation and elementary results}\label{sec:notation}

We study functions $f:\real^d\to\real$.
The argument to $f$ is denoted by 
$\bsx=(x_1,\dots,x_d)$.  The components of $\bsx$
are independent random variables.
We use $\bsx'$ to denote another variable with
the same distribution as $\bsx$, which is independent
of $\bsx$.
We will use hybrid points $\bsx_{-j}{:}\bsx'_j\in\real^d$
that combine inputs from both $\bsx$ and $\bsx'$.
If $\bsy=\bsx_{-j}{:}\bsx'_j$
then
$y_j=x'_j$ and $y_\ell=x_\ell$ for $\ell\ne j$.
We use $[d]$ to denote the set $\{1,2,\dots,d\}$.
For $u\subseteq[d]$ we use $|u|$ for the cardinality
of $u$. 
The point $\bsx_u\in\real^{|u|}$ is comprised of $x_j$ for
$j\in u$. The complement $[d]\setminus u$ is denoted by $-u$ and $\bsx_{-u}$ consists of
those $x_j$ with $j\not\in u$.

\subsection{Radial basis functions}

The description here is based on \cite{fass:2007}.
Radial basis functions are used for scattered data interpolation,
also known as mesh-free approximation, meaning that the sample
points are not necessarily in a regular structure like a grid.
One strong motivation for them is that polynomial interpolation
is not necessarily well defined for an arbitrary
set of points $\bsx_i\in\real^d$ for $d\ge2$ but some RBFs can interpolate at any distinct points. \cite{fass:2007} considers complex valued interpolations
but we consider only real valued functions here.

The RBF interpolant is of the form
$\sum_{i=1}^n\beta_i\phi(\Vert\bsx-\bsx_i\Vert)$.
\cite{fass:2007} considers also the more general
form
$$
\sum_{i=1}^n\beta_i\tilde\phi(\bsx-\bsx_i)
$$
where $\tilde\phi(\cdot)$ is not necessarily
`radial', i.e., not necessarily a function
of the norm of its argument.
To interpolate in this more general setting
we must solve 
\begin{align}\label{eq:tosolve}
K\beta = y
\end{align}
for $\beta\in\real^n$,
where $K\in\real^{n\times n}$ has $ij$ entry $\tilde\phi(\bsx_i-\bsx_j)$ and $y\in\real^n$ has $i$'th entry $f(\bsx_i)$.
The function $\tilde \phi$ is `radial' if $\tilde\phi(\bsx_i-\bsx_j)=\phi(\Vert\bsx_i-\bsx_j\Vert)$ for a function
$\phi:[0,\infty)\to\real$.
The function $\tilde\phi$ is called positive definite if $K$ is always positive semi-definite
for any $n\ge1$
and any distinct points $\bsx_1,\dots,\bsx_n\in\real^d$.  If this $K$ is always positive definite then $\tilde\phi$ is strictly positive
definite.

Strictly positive definite functions can be used to interpolate
any values $f(\bsx_i)$ at distinct $\bsx_i$.  \citet[Chapter 8]{fass:2007}
describes conditionally positive definite functions of order $m$
that can be used to interpolate functions that are orthogonal
to all multivariate polynomials of order less than or equal to $m-1$.
To use them, one interpolates with a suitable polynomial plus
a conditionally positive definite RBF.

\citet[Chapter 3]{fass:2007} provides numerous properties
and characterizations of positive definite functions
and strictly positive definite functions.  If $\tilde\phi$ is 
positive definite then $|\tilde\phi(\bsx)|\le \tilde\phi(0)$.
A real valued continuous and positive definite function
must be even.

Our main interest here is in (strictly) positive definite
radial functions. If $\tilde\phi(\cdot)=\phi(\Vert\cdot\Vert)$ 
is (strictly) positive definite for dimension $d$
then the same holds (strictly or not) for all dimensions $d'\le d$.
Because we want to study the limit as $d\to\infty$ we
are interested in $\phi$ that provide strictly positive definite
functions for all $d\ge1$.
By Theorem 3.8 of \cite{fass:2007}, due to Schoenberg, the function
$\phi:[0,\infty)\to\real$ with
\begin{align}\label{eq:shoenberg}
\phi(r) = \int_0^\infty e^{-r^2t^2}\mu(\mrd t)
\end{align}
provides a strictly positive definite radial function
for all dimensions $d\ge1$
if and only if $\mu$ is a finite positive Borel measure not
concentrated on $\{0\}$.    It follows that these desirable
functions $\phi$ can take no negative values, must be
strictly decreasing, and cannot have compact support.

It is clear from~\eqref{eq:shoenberg} that the
Gaussian RBF $\phi(r)=e^{-r^2\vartheta^2}$ is a strictly
positive definite radial function for $\vartheta>0$
in all dimensions $d\ge1$.
So are generalized inverse multiquadrics
$$
(1+\Vert\bsx\Vert^2)^{p},\quad p<0.
$$
\citep[p 42]{fass:2007}.

Generalized multiquadrics
$(-1)^{\lceil p\rceil}(1+\Vert\bsx\Vert^2)^p$ for 
noninteger $p>0$ are strictly conditionally positive
definite of order $m\ge\lceil p\rceil$.
See \citet[Chapter 8.1]{fass:2007}.
The functions 
$(-1)^{\lceil p/2\rceil}\Vert\bsx\Vert^p$ for $p>0$ with
$p$ not an even integer are conditionally positive definite
of order $m\ge\lceil p/2\rceil$ \citep[Chapter 2.1]{fass:2007}.
By Theorem 9.7 of \cite{fass:2007}, the function
$\phi(\Vert\bsx\Vert)=\Vert\bsx\Vert$ while not strictly
positive definite can be used for interpolation because
the resulting matrix $K$ has one negative and $n-1$ positive
eigenvalues.

Table 3.1 of \cite{fass:mcco:2015} names some of the
more widely used generalized multiquadric RBFs $\phi(r)$:
\begin{align*}
(1+\vartheta^2 r^2)^{-1}&\quad\text{Inverse quadratic}\\
(1+\vartheta^2 r^2)^{-1/2}&\quad\text{Inverse multiquadric}\\
(1+\vartheta^2 r^2)^{1/2}&\quad\text{Multiquadric},
\end{align*}
with a parameter $\vartheta>0$.  The last one is the one
that \cite{hard:1971} uses.

\subsection{ANOVA and Sobol' indices}
We use the standard analysis of variance (ANOVA)
decomposition of $f\in L^2(\real^d)$ from \citep{hoef:1948,sobo:1969,efro:stei:1981}.
This decomposition writes
$$
f(\bsx) = \sum_{u\subseteq[d]}f_u(\bsx)
$$
where the ANOVA effect $f_u$ is a function that
only depends on $\bsx$ through components $x_j$
for $j\in u$. In this decomposition, 
$\e(f_u(\bsx)f_v(\bsx))=0$ for $u\ne v$
and $f_\emptyset$ is the constant function everywhere
equal to $\e(f(\bsx))$.
The quantities
$$
\sigma^2_u = \var( f_u(\bsx))=
\begin{cases}
\e( f_u(\bsx)^2), & u\ne\emptyset\\
0, & u=\emptyset
\end{cases}
$$
are known as the variance components of $f$.
They satisfy $\sigma^2=\sum_{u\subseteq[d]}\sigma^2_u$,
where $\sigma^2=\var(f(\bsx))$.

The unnormalized Sobol' indices of $f$ 
for $u\subseteq[d]$ are
$$
\ult^2_u = \sum_{v\subseteq u}\sigma^2_u
\quad\text{and}\quad
\olt^2_u = \sum_{v:v\cap u\ne\emptyset}\sigma^2_u,
$$
respectively.  Normalized versions 
$\ult^2_u/\sigma^2$ and $\olt^2_u/\sigma^2$
are widely used in global sensitivity
analysis.  See \cite{raza:etal:2021} for context and
an extensive bibliography.
We will use the identity
\begin{align}\label{eq:econdvar}
\olt^2_u = \e( \var( f(\bsx)\giv \bsx_{-u})).
\end{align}
Our greatest need is for $\olt^2_{\{j\}}$ which
we abbreviate to $\olt^2_j$.

When $0<\sigma^2<\infty$, we define the mean dimension of $f$ as 
$$
\nu(f) = \frac1{\sigma^2}\sum_{u\subseteq[d]}\sigma^2_u.
$$
The closest additive function to $f$ in mean square is
$$f_{\mathrm{add}}(\bsx)
=f_\emptyset(\bsx)+\sum_{j=1}^df_{\{j\}}(\bsx).$$
If $\nu(f)$ is close to one
then $f$ is nearly additive in an $L^2$ sense.
More precisely
$$
\nu(f)\le 1+\epsilon \implies
\frac{\var( f(\bsx)-f_{\mathrm{add}}(\bsx))}
{\var(f(\bsx))}\le \epsilon.
$$

An elementary result from \cite{meandim} is that
\begin{align}\label{eq:meandimfromsobol}
\nu(f) = \frac1{\sigma^2}\sum_{j=1}^d\olt^2_j.
\end{align}
\cite{jans:1999} has a useful identity
\begin{align}\label{eq:jansens}
\olt^2_j
=\frac12\e\bigl(
(f(\bsx_{-j}{:}\bsx'_j)-f(\bsx))^2
\bigr)
\end{align}
that allows sampling based estimates of
$\olt^2_j$. This identity underlies our
theoretical analysis along with the more familiar
identity $\sigma^2=\e( (f(\bsx)-f(\bsx'))^2)/2$.

\section{Generalized multiquadrics}\label{sec:meandimmultiq}

These functions take the form
$(a+\vartheta\Vert\bsx\Vert^2)^p$.
We can rewrite them as
$(a+\Vert\bsx\Vert^2)^p$ after
replacing $a$ by $a/\vartheta$ and
rescaling the coefficients $\beta_i$ by a factor of $\vartheta^p$.
The cases that interest us most have
nonzero $p<1$ because those get the most use.
The case $p=1$ is obviously of mean dimension
one.  
We will include cases with $a=0$ and $p<0$.
As \cite{fass:2007} notes, these are not well suited to interpolation
due to their singularities but they are of interest as
generalized Coulomb potentials.

\subsection{Parametrization of generalized
multiquadrics}
A radial basis function uses the inputs
$\bsx$ only through $\sum_{j=1}^d(x_j-c_j)^2$.
Here $x_j$ is the $j$'th component of $\bsx$ and
$c_j$ is the $j$'th component of a center point such as $\bsx_i$.
We let
$$
z_j = 
\begin{cases}
a + (x_1-c_1)^2, & j=1\\
(x_j-c_j)^2,& \text{else}  
\end{cases}
$$
and then we study mean dimension in terms
of random $\bsz = (z_1,\dots,z_d)$.
We have folded any $a>0$ into $z_1$ to remove $a$ from further
expressions.  The case of $a=0$ is the most challenging
because it can produce a singularity at $\bsz=0$ that we don't have
to consider when $a>0$.

The radial basis functions we study are functions
of $\bsz$ where $\bsz$ is defined componentwise
from $\bsx$.  If we use $f^*$ to represent
the radial basis function in terms of $\bsz$
then we find the same mean and variance and variance
components and mean dimension for $f^*$
as we get for $f$.
For simplicity, we will use $f$ also for the radial
basis function written in terms of $\bsz\in[0,\infty)^d$.
We retain the distinction between $\bsx$ and $\bsz$
because that makes our input assumptions
easier to interpret. 
We will study the mean dimension of $\bigl(\sum_{j=1}^dz_j\bigr)^p$
for independent not necessarily identically distributed random $z_j\ge0$ and nonzero $p\le1$.

\subsection{Assumptions on $\bsz$}
We study a collection of independent nonnegative random
variables $z_j$ for $j=1,\dots,d$.
We write $\mu_j=\e(z_j)$ and $\sigma^2_j =\var(z_j)$.
Some higher moments are denoted by $\mu^{(k)}_j=\e( (z_j-\mu_j)^k$
for positive integers $k$.
For certain sums we write
$$
z_{1:d} :=\sum_{j=1}^dz_j,\quad\mu_{1:d}:=\sum_{j=1}^d\mu_j,
\quad\sigma^2_{1:d} := \sum_{j=1}^d\sigma^2_j
\quad\text{and}\quad \mu^{(k)}_{1:d} := \sum_{j=1}^d\mu^{(k)}_j.
$$
We want to bound the mean dimension of $(z_{1:d})^p$.
It is convenient to define
\begin{align}\label{eq:ourf}
f(\bsz) = \Bigl(\frac{z_{1:d}}{\mu_{1:d}}\Bigr)^p.
\end{align}
This function of $\bsz$ has the same mean dimension as if 
we had not scaled the input by $\mu_{1:d}$ and it has
the same mean dimension as the original function of $\bsx$.

We will use a bounded mean assumption
\begin{align}\label{eq:boundedmean}
0<\umu \le \mu_j \le \omu <\infty, \quad 1\le j\le d
\end{align}
and a bounded variance assumption
\begin{align}\label{eq:boundedvar}
0<\sigma^2_j \le \osigsq <\infty, \quad 1\le j\le d
\end{align}
and for some $\alpha>0$, a negative moment assumption
\begin{align}\label{eq:negmoment}
\e( z_j^{-\alpha})\le M_\alpha<\infty, \quad 1\le j\le d.
\end{align}
For some of our sharper result we will require  that 
\begin{align}\label{eq:sixmomentbounds}
|\e( (z_j-\mu_j)^k)|\le \lambda\quad\text{for $2\le k\le 6$}
\end{align}
holds for some $\lambda <\infty$.

We do not lose much generality requiring $\sigma^2_j>0$
because $\sigma^2_j=0$ implies that $z_j$ is redundant.
Our main results will still hold 
if some $\sigma^2_j=0$ so long as $\sigma^2_{1:d}>0$.

One very important case has $x_j\sim\dnorm(0,1)$.
Then $z_j$ has a noncentral chi-squared distribution
with one degree of freedom and noncentrality parameter $c_j^2$.
This distribution satisfies the bounded mean and variance
assumptions provided that $c_j^2$ is bounded.
It satisfies the negative moment assumption if $\alpha<2$
because the central $\chi^2_{(1)}$ satisfies that condition
and the noncentral distribution is a mixture of
central $\chi^2$ distributions with odd numbers of
degrees of freedom.
For the case with finite $a>0$, $z_1$ satisfies these three conditions if
$(x_1-c_1)^2$ does.

\subsection{Main result for generalized multiquadrics}

Here we present our main result for mean
dimension of generalized multiquadric RBFs.
We give moment conditions on $z_j$ under
which 
\begin{align}\label{eq:nubound}
\nu(f) \le 1+\frac{(p-1)^2}2\frac{\sigma^2_{1:d}}{(\mu_{1:d})^2}+O\Bigl(\frac1{d^2}\Bigr).
\end{align}
Most of the proof details are
in Appendix~\ref{sec:rates}.
We assume throughout that independent $z_j\ge0$
satisfy the sixth moment condition~\eqref{eq:sixmomentbounds},
the mean condition~\eqref{eq:boundedmean}
and the negative moment condition
\eqref{eq:negmoment}.
The results in Appendix~\ref{sec:rates} depend on some results in Appendix~\ref{sec:summoments} about
positive and negative moments of sums of~$z_j$.

The mean dimension equals
$\sum_{j=1}^d\olt^2_j/\sigma^2$, so
we get asymptotic expressions for the numerator
and denominator of this ratio.
For the denominator,
Proposition~\ref{prop:pmomentwith6} 
in Appendix~\ref{sec:rates} shows
that
$\e( ({z_{1:p}}/{\mu_{1:p}})^p)$ equals 
$$1
+ \biggl( \frac{ (p)_2 }{2!} \cdot \frac{ \sigma^2_{1:d} }{ ( \mu_{1:d})^2} \biggr) 
+ \biggl( \frac{ (p)_3 }{3!} \cdot \frac{ \mu^{(3)}_{1:d} }{ (\mu_{1:d})^3 } 
+ \frac{ (p)_4 }{4!} \cdot \frac{3 (\sigma^2_{1:d})^2 }{ (\mu_{1:d})^4 } \biggr)
+ O(d^{-3})
$$
for $p<6$ as $d\to\infty$. Here
$(p)_k=p(p-1)\cdots(p-k+1)$.
Using this result for first and
second moments of $(z_{1:d}/\mu_{1:d})^p$
for $p<1$,
Corollary~\ref{cor:varfforrate}
shows that 
$\sigma^2= \var\left( \left( {z_{1:d}}/{\mu_{1:d}} \right)^p \right)$
equals
$$\frac{ p^2 \sigma_{1:d}^2 }{ (\mu_{1:d})^2 } \Bigl( 
	1
	+ (p-1) \cdot \frac{ \mu^{(3)}_{1:d} }{ \sigma^2_{1:d} \mu_{1:d} } 
	+ \frac{1}{2} (p-1)(3p-5) \cdot \frac{\sigma^2_{1:d}}{(\mu_{1:d})^2} 
	+ O(d^{-2}) \Bigr). 
$$
For the numerator, Proposition~\ref{prop:upperboundsumtausq}
shows that 
$$ \sum_{j=1}^d \olt_j^2 
\leq \frac{p^2 \sigma_{1:d}^2}{(\mu_{1:d})^2} \Bigl( 1 + (p-1)(2p-3) \frac{\sigma_{1:d}^2}{(\mu_{1:d})^2}
+ (p-1) \frac{ \mu^{(3)}_{1:d}}{ \mu_{1:d} \sigma_{1:d}^2 } + O(d^{-2}) \Bigr).$$

\begin{theorem}\label{thm:asyboundonnu}
Let $z_j\ge0$ be independent random variables satisfying
the sixth moment condition~\eqref{eq:sixmomentbounds}
and the mean condition~\eqref{eq:boundedmean}.
Let $f(\bsz)=(z_{1:d}/\mu_{1:d})^p$ for nonzero
$p<1$. Then
$$
\nu(f) \le 1+\frac{(p-1)^2}2\frac{\sigma^2_{1:d}}{(\mu_{1:d})^2}+O\Bigl(\frac1{d^2}\Bigr)
$$
as $d\to\infty$.
\end{theorem}
\begin{proof}
Combining the upper bound from Proposition~\ref{prop:upperboundsumtausq}
with the asymptotic variance in Corollary~\ref{cor:varfforrate}
we get
\begin{align*}
\nu(f) &\le
\frac{
1+(p-1)(2p-3)\frac{\sigma^2_{1:d}}{(\mu_{1:d})^2}
+(p-1)\frac{\mu_{1:d}^{(3)}}{\mu_{1:d}\sigma^2_{1:d}}
+O(d^{-2})}
{1 
+ \frac{1}{2} (p-1)(3p-5) \frac{\sigma^2_{1:d}}{(\mu_{1:d})^2} 	+ (p-1) \frac{ \mu^{(3)}_{1:d} }{ \sigma^2_{1:d} \mu_{1:d} }+ O(d^{-2}) 
}\\
&=1 +\frac{(p-1)^2}2\frac{\sigma^2_{1:d}}{(\mu_{1:d})^2}
+O\Bigl( \frac1{d^2}\Bigr). \qedhere
\end{align*}
\end{proof}

\begin{remark}
Under the assumptions we have made,
$\sigma^2_{1:d}/(\mu_{1:d})^2=\Theta(1/d)$.
\end{remark}
\begin{remark}
We notice that the bound in Theorem~\ref{thm:asyboundonnu} can
be evaluated for the degenerate case $p=0$.
We conjecture that this might be the rate for
$f(\bsz)=\log( z_{1:d})$.
Our reasoning is that the mean dimension of
$(z_{1:d})^p$ is the same as that of
$((z_{1:d})^p-1)/p$ which approaches $\log(z_{1:d})$
as $p\to0$.
\end{remark}

\subsection{Weaker conditions}

Theorem~\ref{thm:asyboundonnu}
relies on a sixth moment assumption in order to get an
expression for the coefficient
of $1/d$ in the bound on $\nu(f)$.
This section shows that the mean dimension of
generalized multiquadric RBFs tends to $1
$ as $d\to\infty$
under very mild moment conditions:
means and variances of $z_j$  bounded
uniformly from $0$ and $\infty$ and
a finite negative moment.
Under these conditions, Lemmas~\ref{lem:sobindexupper}
and \ref{lem:varlowerbd}
in Appendix~\ref{sec:boundednu}
show that
$$ 
\limsup_{d \rightarrow \infty} \frac{ \sum_{j=1}^d \olt_j^2 }{ p^2 \cdot \frac{\sigma_{1:d}^2 }{ (\mu_{1:d})^2 } }
\leq 1
\quad\text{and}\quad
 \liminf_{d \rightarrow \infty} 
 \frac{\var(f(\bsz))}{p^2 \cdot \frac{\sigma_{1:d}^2}{ (\mu_{1:d})^2 } }\geq 1$$
 respectively.

\begin{theorem}\label{thm:limisone}
Let independent random $z_j\ge0$ for $j=1,\dots,d$ satisfy the
mean bounds~\eqref{eq:boundedmean},
the variance bounds~\eqref{eq:boundedvar} and the negative moment condition~\eqref{eq:negmoment}.
Let $f(\bsz)=(z_{1:d}/\mu_{1:d})^p$ for non-zero $p< 1$.
Then the mean dimension of $f$ satisfies
$$
\lim_{d\to\infty}\nu(f) = 1.
$$
\end{theorem}
\begin{proof}
By definition $\nu(f)\ge1$. Next
\begin{align*}
\lim_{d\to\infty}\nu(f) &=
\lim_{d\to\infty}\frac{\sum_{j=1}^d\olt^2_j}{\var(f(\bsz))}
\le \frac{\limsup_{d\to\infty} \sum_{j=1}^d\olt^2_j
/[p^2\sigma^2_{1:d}/(\mu_{1:d})^2]}
{\liminf_{d\to\infty}\var(f(\bsz))/[p^2\sigma^2_{1:d}/(\mu_{1:d})^2]}
\end{align*}
which equals 1 by Lemmas~\ref{lem:sobindexupper} and~\ref{lem:varlowerbd}.
\end{proof}

\section{The Gaussian RBF}\label{sec:gaussian}
Here we show how the Gaussian RBF is not limited to low mean dimension as
$d\to\infty$ because the scale parameter can be chosen to control
mean dimension.  This makes it very different from multiquadric and
related RBFs where the asymptotic mean dimension must converge to one.
The Gaussian RBF is special in that it 
can be parameterized as a product
$$
f(\bsx) = \prod_{j=1}^d\exp(-(x_j-c_j)^2/\vartheta^2)
$$
for $\vartheta>0$.  
We have changed the scaling from $(x_j-c_j)^2\vartheta^2$ to $(x_j-c_j)/\vartheta^2$
to give $\vartheta^2$ an interpretation as
twice the variance of a Gaussian random variable.
We assume that $x_j$ are independent
with  a continuous
distribution.  Without loss of generality
we assume that $x_j$ have mean zero.

We use three propositions.
The product form of the Gaussian
RBF allows for a simplification
of the mean dimension.  Proposition~\ref{prop:meandimprod} below applies
to general products, not just Gaussian RBFs.

\begin{proposition}\label{prop:meandimprod}
Let $f(\bsx)=\prod_{j=1}^dg_j(x_j)$ where $x_j$ are independent
random variables with $\var(g_j(x_j))<\infty$ and $\min_{1\le j\le d}\var(g_j(x_j))>0$.  Then $f$ has mean dimension
\begin{align}\label{eq:meandimfromlambda}
\nu(f)=\frac{\sum_{j=1}^d\rho_j}{1-\prod_{j=1}^d(1-\rho_j)}
\end{align}
where
$$
\rho_j = \frac{\var(z_j)}{\e(z_j^2)}\in[0,1].
$$
\end{proposition}
\begin{proof}
This is Proposition 1 of \cite{dimdist}.
\end{proof}

\begin{proposition}\label{prop:monotoneinrho}
Under the conditions of Proposition~\ref{prop:meandimprod},
$$
\frac{\partial}{\partial\rho_k}\nu(f)\ge0.
$$
\end{proposition}
\begin{proof}
The result is trivial for $d=1$, so we assume that $d\ge2$.
The partial derivative is
$$
\frac{1-\prod_{j\ne k}(1-\rho_j)[1+\sum_{j\ne k}\rho_j]}{[1-\prod_{j=1}^d(1-\rho_j)]^2}.
$$
The denominator above is positive.  
Letting $\bar\rho_{-k}=(d-1)^{-1}\sum_{j\ne k}\rho_j$,
the numerator is
at least
\begin{align}\label{eq:numbound}
1-(1-\bar\rho_{-k})^{d-1}(1+(d-1)\bar\rho_{-k})
\end{align}
because the geometric mean of $1-\rho_j$ for $j\ne k$ is no
larger than their arithmetic mean.
The expression in~\eqref{eq:numbound} is increasing in
$\bar\rho_{-k}$ over $\bar\rho_{-k}\in(0,1)$
and it equals zero for $\bar\rho_{-k}=0$.
\end{proof}

\begin{proposition}\label{prop:gausrbflimits}
Let $x$ be a random variable with probability
density function $h$ on $\real$.  Assume that
$h(x)\le M$ and that $c\in\real$ belongs to an
interval $I$ of length at least $\ell>0$ on
which $h(x)\ge h_0>0$.  Then
\begin{align*}
\lim_{\vartheta\to\infty}\frac{\e(e^{-2(x-c)^2/\vartheta^2})}{\e(e^{-(x-c)^2/\vartheta^2})^2 }=1.
\end{align*}
and
\begin{align}\label{eq:gausmsebysquaremean}
\lim_{\vartheta\downarrow0}\frac{\e(e^{-2(x-c)^2/\vartheta^2})}{\e(e^{-(x-c)^2/\vartheta^2})^2 }=\infty.
\end{align}
\end{proposition}
\begin{remark}
The first limit has a mean square over a squared
mean approach 1.  Then the variance becomes negligible,
so $\rho\to0$ in the above notation. The second limit
has a mean square divided by a squared mean approach
infinity, so $\rho\to1$ in the above notation.
\end{remark}
\begin{proof}
The first claim is easy as both numerator and denominator
approach $1$ as $\vartheta\to\infty$.  For the second claim
\begin{align*}
\e(e^{-(x-c)^2/\vartheta^2})
&=\int_{-\infty}^\infty
e^{-(x-c)^2/\vartheta^2}
h(x)\rd x\le M\sqrt{\pi}\vartheta.
\end{align*}
We let $I = (a,b)$ with $b-a=\ell$.
Next by change of variable
\begin{align*}
   \e(e^{-2(x-c)^2/\vartheta^2})
&=\frac{\vartheta}2\int_{-\infty}^\infty e^{-y^2/2}h( c+\vartheta y/2)\rd y\\
&\ge\frac{\vartheta h_0}2\int_{-\infty}^\infty e^{-y^2/2}\bsone_{c+\vartheta y/2\in I}\rd y\\
&=\frac{\vartheta h_0}2\int_{2(a-c)/\vartheta}^{2(b-c)/\vartheta} e^{-y^2/2}\rd y\\
&=\vartheta h_0\sqrt{\frac\pi2}
\Bigl(\Phi\Bigl(\frac{2(b-c)}\vartheta\Bigr)-\Phi\Bigl(\frac{2(a-c)}{\vartheta}\Bigr)\Bigr) \\
&\ge\vartheta h_0\sqrt{\frac\pi2}
\Bigl(\Phi\Bigl(\frac{2(b-a)}\vartheta\Bigr)-\frac12\Bigr). \end{align*}
For any $\ell=b-a>0$ we can choose $\vartheta$ small enough to make $\Phi( 2\ell/\vartheta)\ge3/4$
and then $\e(e^{-2(x-c)^2/\vartheta^2})\ge \vartheta h_0\sqrt{\pi/32}$. Now the numerator in~\eqref{eq:gausmsebysquaremean} is $\Omega(\vartheta)$
while the denominator is $O(\vartheta^2)$ both as
$\vartheta\to\infty$. The result follows.
\end{proof}

In the Gaussian setting, $\rho_j>0$ and ruling out
uninteresting variables with $\var(x_j)=0$ we also
have $\rho_j<1$.
The mean dimension of $f$ is continuous and
nondecreasing in each $\rho_j$,
by Proposition~\ref{prop:monotoneinrho}.
By Proposition~\ref{prop:gausrbflimits},
each $\rho_j\to1$ as $\vartheta\to0$, 
when $x_j$ has a continuous distribution
and so $\nu(f)\to d$.
Conversely as $\vartheta\to\infty$, each $\rho_j\to0$
and then $\nu(f)\to1$. 
Therefore any mean dimension in
$(1,d)$ can be attained at some value of $\vartheta$.

\section{Keister's function}\label{sec:keister}

The Keister function was used by \cite{keis:1996}
and also \cite{caps:keis:1996},
to compare multi-dimensional quadrature methods.
These papers use $\int_{\real^d}e^{-\Vert\bsx\Vert^2}\cos(\Vert\bsx\Vert)\rd\bsx$
as an example of the sort of integration problem arising in atomic, nuclear and particle physics.
We make a change of variable and consider
$$
f(\bsx) = \cos( \Vert\bsx\Vert/2),
$$
for $\bsx\sim\dnorm(0,I)$.
This $f$ is a radial basis function but not one of those commonly used
for approximation. \cite{caps:keis:1996} and \cite{keis:1996}
give precise values for $\e(f(\bsx))$ at certain values of $d$.
\cite{jaga:hick:2019} give a recursion for this expectation.

Keister's function has become a test
function for QMC, since  \cite{papa:trau:1997}.
The success of QMC on some integrands from
finance could possibly be explained by the
unequal importance of the variables in those
integrands.  Perhaps many of them were
quite unimportant leaving an integrand that
depends on only a few variables.  All $d$
variables enter Keister's function symmetrically
so there would need to be another explanation
for QMC's successes there.
The explanation is that it is dominated by its
low dimensional ANOVA components.  Computations
in \cite{dimdist} show that for $d=25$ 
(the dimension considered by \cite{papa:trau:1997}
and $d=80$, over $99$\% of the variance of the Keister function
comes from variance components $\sigma^2_u$ with
$|u|\le 3$ making it of effective dimension
$3$ in the sense of \cite{cafl:moro:owen:1997}.
Here we study the Keister function's mean dimension
for $2\le d\le1000$.

By symmetry, $\olt_1^2=\olt^2_2=\cdots=\olt^2_d$ for the Keister function and
so its mean dimension is
$\nu(f)=d\olt^2_1/\sigma^2$.
The variance $\sigma^2$ can easily
be approximated by sampling because $\Vert\bsx\Vert^2\sim\chi^2_{(d)}$.
For this paper, we used a midpoint
rule on $n=2^{14}=16{,}384$ points
in $(0,1)$, transformed them to $\chi^2_{(d)}$
quantiles, took the square root
to get a sample value for 
$\Vert\bsx\Vert$
and then computed the sample variance
of the $\cos(\Vert\bsx\Vert/2)$ values.

To estimate $\olt^2_1$, we find using
the Jansen identity \eqref{eq:jansens} 
that
$$
\olt^2_1 = \frac12\e\bigl((f(z_1+z_2)-f(z_1+z_3) )^2\bigr)
$$
for
$z_1=\sum_{j=2}^dx_j^2$,
$z_2=x_1^2$ and $z_3={x'_1}^2$.
Now $z_1\sim\chi^2_{(d-1)}$,
$z_2\sim\chi^2_{(1)}$ and
$z_3\sim\chi^2_{(2)}$ are independent
random variables.  We then estimated
$\olt^2_1$ by using randomized Sobol'
points in $(0,1)^3$, transforming them to the needed $\chi^2$ values by inversion of their cumulative distribution functions and
applied the Jansen formula.
For this integral we used a
Sobol' sequence \cite{sobo:1967:tran}
with direction numbers from~\cite{joe:kuo:2008}
and a nested uniform scramble  of \cite{rtms}
with $n=2^{14}=16{,}384$ points.

The above computation was replicated
five times independently.
With a bit of foresight, we plot mean dimension of Keister's
function in dimension $d$ versus $\sqrt{d}$ in Figure~\ref{fig:keistersqrt}.  
The plot shows all five replicates but
they overlap each other in the figure.
The mean dimension is
not monotone in $d$. Instead for $d\ge2$, the mean dimension oscillates regularly from
just over $1$ to peaks that are eventually just over $2$.

\begin{figure}
\centering
\includegraphics[width=\hsize]{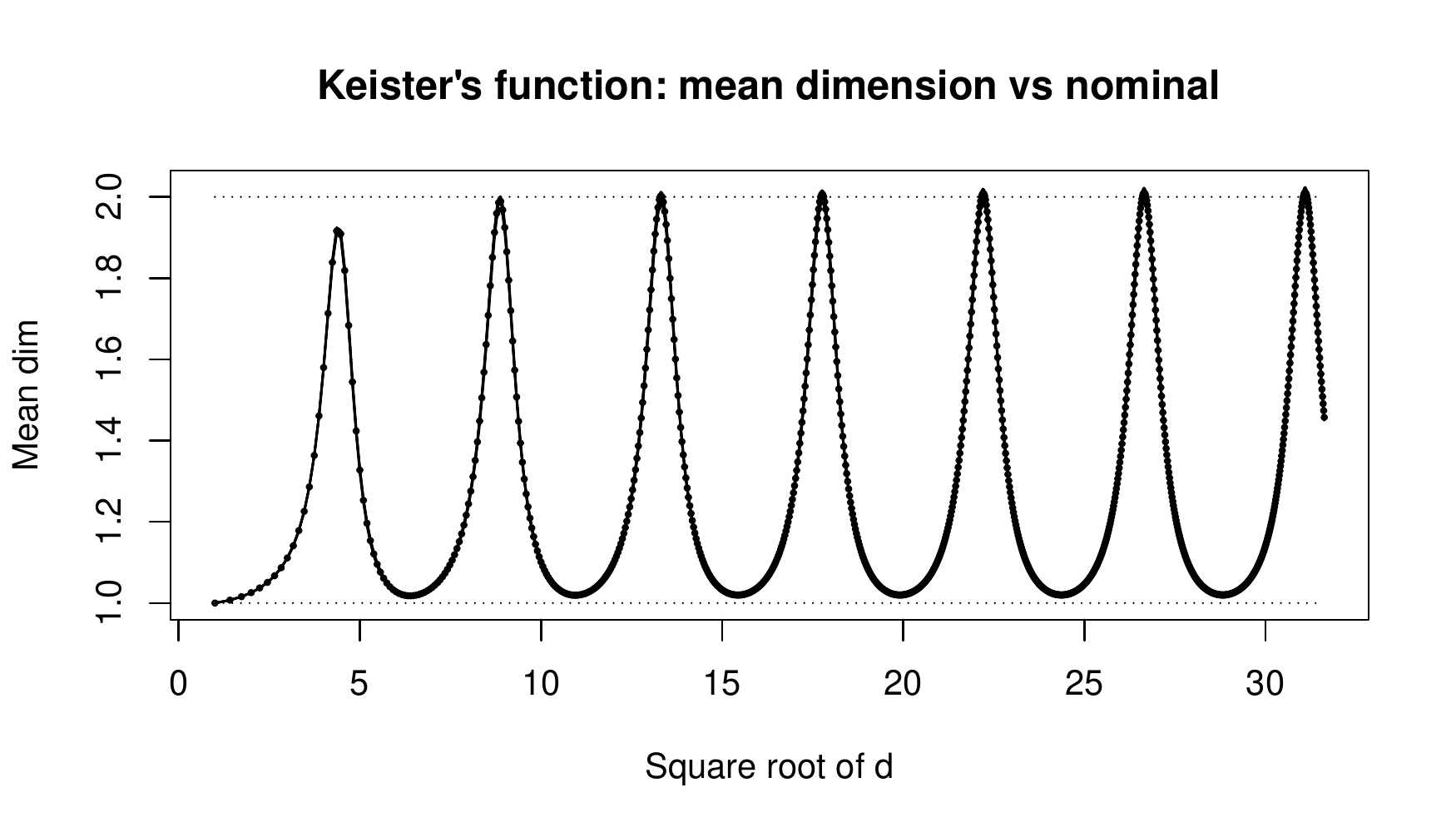}
\caption{\label{fig:keistersqrt}
The horizontal axis has $\sqrt{d}$ for $1\le d\le1000$.
The vertical axis plots five lines, each an independent
randomized QMC estimate of mean dimension versus $\sqrt{d}$.  Points
mark the average of the five values. There are dotted horizontal references lines
at levels $1$ and $2$.
}
\end{figure}

From Figure~\ref{fig:keistersqrt} it becomes clear what is going on.
The random variable $\Vert\bsx\Vert^2$ has a $\chi^2_{(d)}$
distribution.  For large $d$, this is approximately $\dnorm(d,2d)$.
Then by the delta method (Taylor approximation about the mean),
$\Vert\bsx\Vert/2$ has approximately the $\dnorm(\sqrt{d}/2,1/4)$
distribution.  The central $99.9$\% of $\dnorm(\alpha,1/4)$ values
belong to the range $\alpha\pm\Phi^{-1}(0.9995)/\sqrt{4}$
or about $\alpha\pm1.65$. Then $\cos(\Vert\bsx\Vert/2)$
primarily uses the cosine function over an interval of length about $3.3$,
roughly half of the period $2\pi$ of the cosine function. When $\sqrt{d}/2$
is nearly an integer multiple of $\pi$, then the cosine function is being
sampled predominantly in a region where it is nearly quadratic
and we find that the mean dimension is close to $2$.
If instead, $\sqrt{d}/2$ is nearly $\pi/2$ plus an integer multiple of $\pi$,
then the cosine is being sampled over a nearly linear range and the mean
dimension is close to $1$.

\section{Discussion}\label{sec:conclusions}

Much success in high dimensional numerical
methods comes from the target function having less
complexity than we might expect given its
nominal dimension.  See \cite{cafl:moro:owen:1997}
or \cite{nova:wozn:2008} or \cite{sloa:wozn:1998}
among other references. In that literature,
tractability results provide sets of assumptions
under which there is no curse of dimensionality.
\cite{effdimsobolev} and references
therein show that
some weighted Hilbert spaces for which dimension-independent tractability has been established
have very low effective dimension in the
superposition sense (e.g., $3$ or less
from the $\eta=1$ column of Table 1 in \cite{effdimsobolev}).  An effective
dimension of $3$ implies that there are
only negligible contributions to $f$
from variance components $\sigma^2_u$
with $|u|>3$.
We then expect that 
components $\phi(\bsx)$ with a mean 
dimension that is $O(1)$ as $d\to\infty$
to be most useful because they fit naturally with the
subset of high dimensional problems where tractability
results have been established.

Generalized multiquadric RBFs with a mean dimension
of $1+O(1/d)$ take this too far because they will have
difficulty with test problems involving even two or three
factor interactions.  Gaussian RBFs can attain such
mean dimensions if their parameters are well chosen.
Ridge functions $\phi(\theta^\tran\bsx)$
with Lipshitz continuous
$\phi(\cdot)$ and a unit vector $\theta$
attain an $O(1)$ mean dimension automatically,
under Gaussian sampling \citep{hoyt:owen:2020}.

\section*{Acknowledgments}
We thank Naofumi Hama for comments on the
role of RBFs in machine learning.
This work was supported by the U.S.\ National Science Foundation under
grants  IIS-1837931  
and DMS-2152780 and by Hitachi, Ltd.
\bibliographystyle{plainnat}
\bibliography{qmc}

\appendix

\section{Moments of some sums}
\label{sec:summoments}
Here we provide some moment formulas
needed later.
We begin by working out some expressions for
central moments of sums of our random variables.
For integers $k\ge2$ we use
$\mu^{(k)}_j=\e((z_j-\mu_j)^k)$ to denote
$k$'th central moments and
$$\mu_{1:d}^{(k)}\equiv\sum_{j=1}^d \mu_j^{(k)}.$$
For $k=1$, we use $\mu_j=\e(z_j)$ and $\mu_{1:d}=\sum_{j=1}^d\mu_j$
and for $k=2$, we use $\sigma^2_j$ and
$\sigma^2_{1:d}=\sum_{j=1}^d\sigma^2_j$.
The following theorem simplifies some derivations.
\begin{theorem}\label{thm:petrovs}
For $d\ge1$, let $x_1,\dots,x_d$ be independent random variables with $\e(|x_j|^k)<\infty$
for $j=1,\dots,d$ and some integer $k\ge2$.
Set $x_{1:d}=\sum_{j=1}^dx_j$.
If $\e(x_j)=0$ for $j=1,\dots,d$ then
$$
\e[ |x_{1:d}|^k] \le c(p)d^{k/2-1}\sum_{j=1}^d\e[ |x_j|^k]
$$
for some $c(k)<\infty$.
\end{theorem}
\begin{proof}
This is in \cite{petrov1992moments}.
\end{proof}

\begin{proposition}\label{prop:sixmoments}
For $j=1,\dots,d$, let $z_j$ be independent random variables
with means $\mu_j$ and variances $\sigma^2_j$.
Let $z_1,\dots,z_d$ satisfy the sixth moment bounds~\eqref{eq:sixmomentbounds}
for some $\lambda<\infty$.
Then
\begin{align}
\e[ (z_{1:d} / \mu_{1:d} - 1)^2 ] &= \sigma_{1:d}^2 / (\mu_{1:d})^2 \label{eq:mom2}\\
\e[ (z_{1:d} / \mu_{1:d} - 1)^3 ] &= \mu^{(3)}_{1:d} / (\mu_{1:d})^3 \label{eq:mom3}\\
\e[ (z_{1:d} / \mu_{1:d} - 1)^4 ] &= 3 (\sigma_{1:d}^2)^2 / (\mu_{1:d})^4 + O(d^{-3}) \label{eq:mom4}\\
\e[ (z_{1:d} / \mu_{1:d} - 1)^5 ] &= O(d^{-3}),\quad\text{and} \label{eq:mom5}\\
\e[ (z_{1:d} / \mu_{1:d} - 1)^6 ] &= O(d^{-3}).\label{eq:mom6} 
\end{align}
\end{proposition}

\begin{proof}
The results for exponents $k=2,3,4$ are elementary.
Theorem~\ref{thm:petrovs} (Petrov) yields
$$\e\Bigl[ \Bigl(\frac{z_{1:d}}{\mu_{1:d}}-1\Bigr)^k\Bigr]
\le c(k)d^{k/2-1}
\frac{\mu^{(k)}_{1:d}}{(\mu_{1:d})^k}
\le \frac{\lambda c(k)d^{k/2}}{(d\umu)^k}
=O(d^{-k/2}).
$$
Taking $k=6$ above provides the result~\eqref{eq:mom6}
for the sixth moment.

The case of $k=5$
remains. Petrov's Theorem would only give us
$O(d^{-5/2})$.  The difference is that Petrov's
theorem is about an absolute moment and
our requirement is for just for an expected
fifth power.
For $k=5$ we get
\begin{align}
\e[ (z_{1:d}-\mu_{1:d})^5]
&=
10\sum_{\substack{j_1,j_2\in[d]\\\text{distinct}}}
\e[ (z_{j_1}-\mu_{j_1})^3(z_{j_2}-\mu_{j_2})^2]
+\sum_{j=1}^d\e[ (z_j-\mu_j)^5]\notag\\
&=10\mu^{(3)}_{1:d}\sigma^2_{1:d}-10\sum_{j=1}^d\mu^{(3)}_j\sigma^2_j+\mu^{(5)}_{1:d},\label{eq:5thing}
\end{align}
where the factor $10$ comes from
there being $10$ partitions like
$j_1=j_2\ne j_3=j_4=j_5$.
The quantity in~\eqref{eq:5thing}
is then $O(d^2)$
establishing~\eqref{eq:mom5}.

The implied constant in the fourth
degree term can be taken as $3\lambda$.
The implied constant in the fifth
degree term can be taken as 
$10\lambda^2/\umu^5+\epsilon$
for any $\epsilon>0$.
The implied constant in the sixth
degree term can be taken as
$\lambda c(6)/\umu^6$. 
\end{proof}

For the next result, we prove an upper bound on negative moments. We use the quantity
\begin{align}\label{eq:quantitybeta}
    \beta := \frac{1}{ \umu \sqrt[\alpha]{M_\alpha} }
\end{align}
recalling that $\e(z_j^{-\alpha})\le M_\alpha<\infty$ 
from equation~\eqref{eq:negmoment}.
This $\beta$ is useful in providing constant upper bounds for negative moments.

\begin{proposition}\label{prop:constUpperBound}
Let $z_1, z_2, \ldots, z_d$ be independent 
nonnegative random variables that satisfy the mean bounds~\eqref{eq:boundedmean} and the negative moment assumption~\eqref{eq:negmoment} for some $\alpha>0$ and choose an exponent $p<0$. Then $\e( (z_{1:d}/\mu_{1:d})^p)\le\beta^p$ for all $d \geq -p/\alpha$.
\end{proposition}
\begin{proof}

For $d \geq -p/\alpha$, we find that $\phi(x) = x^{-\alpha d/p}$ is a convex function. Then
using the mean lower bound~\eqref{eq:boundedmean}, the arithmetic-geometric mean identity
and Jensen's inequality
\begin{align*}
\mathbb{E}\left[ \left( \frac{z_{1:d}}{\mu_{1:d}} \right)^{p} \right] 
&= \left( \frac{d}{ \mu_{1:d} } \right)^p \cdot \mathbb{E}\left[ \left( \frac{z_{1:d}}{d} \right)^{p} \right] \\
&\leq \left( \frac{d}{ d \umu } \right)^p \cdot \mathbb{E}\left[ z_1^{p/d} z_2^{p/d} \cdots z_d^{p/d} \right] 
\\
&\leq \umu^{-p} \cdot \mathbb{E}\left[ (z_1^{p/d} z_2^{p/d} \cdots z_d^{p/d})^{-\alpha d / p} \right]^{-p / \alpha d} 
\\
&\leq \umu^{-p} \cdot \mathbb{E}\left[ z_1^{-\alpha } z_2^{-\alpha} \cdots z_d^{-\alpha} \right]^{-p / \alpha d} \\
&\leq \umu^{-p} \cdot (M_\alpha^d)^{-p / \alpha d} \\
&\leq \beta^p.\qedhere
\end{align*}
\end{proof}

\begin{remark}
This result shows that any
negative moment of the sample average
is $O(1)$ as $d\to\infty$, under the
given conditions.
\end{remark}

This next result is used to control the Lagrange error
term in some Taylor approximations.
\begin{proposition}\label{prop:pmomentwith6}
Let independent $z_j\ge0$ satisfy the mean
bounds~\eqref{eq:boundedmean} as well as
condition~\eqref{eq:sixmomentbounds} on their
first six central moments
and the negative moment
condition~\eqref{eq:negmoment}.
Then for $p<6$, 
$\e( (z_{1:d}/\mu_{1:d})^p)$ equals
$$ 
1
+ \biggl( \frac{ (p)_2 }{2!} \cdot \frac{ \sigma^2_{1:d} }{ ( \mu_{1:d})^2} \biggr) 
+ \biggl( \frac{ (p)_3 }{3!} \cdot \frac{ \mu^{(3)}_{1:d} }{ (\mu_{1:d})^3 } 
+ \frac{ (p)_4 }{4!} \cdot \frac{3 (\sigma^2_{1:d})^2 }{ (\mu_{1:d})^4 } \biggr)
+ O(d^{-3})
$$
as $d\to\infty$.
\end{proposition}
\begin{proof} 
Using a fifth order Taylor expansion
we get
\begin{align*}
\Bigl( \frac{z_{1:d}}{\mu_{1:d}} \Bigr)^p
&= \sum_{k=0}^5 \frac{ (p)_k }{ k! } \Bigl( \frac{z_{1:d}}{\mu_{1:d}} - 1 \Bigr)^k
+ \frac{ (p)_6 }{ 6! } \Bigl( \frac{z_{1:d}}{\mu_{1:d}} - 1 \Bigr)^6 \cdot  \theta^{p-6}
\end{align*}
for some $\theta$ between 1 and $z_{1:d} / \mu_{1:d} $.
Using the results in Proposition~\ref{prop:sixmoments}
we find that the expected value of the 
sum for $0\le k\le 5$ is
\begin{align*}
	1
	+ \Bigl( \frac{ (p)_2 }{2!} \cdot \frac{ \sigma_{1:d}^2 }{ (\mu_{1:d})^2} \Bigr) 
	+ \Bigl( \frac{ (p)_3 }{3!} \cdot \frac{ \mu^{(3)}_{1:d} }{ (\mu_{1:d})^3} 
	+ \frac{ (p)_4 }{4!} \cdot\frac{3 (\sigma_{1:d}^2)^2 }{(\mu_{1:d})^4} \Bigr) +O(d^{-3}).
\end{align*}
It remains to show that the remainder
term with $k=6$ is $O(d^{-3})$.

We can assume that $d\ge2$ and then define
$$ A_d = \frac{ \sum_{j \in [d], \text{even}} z_j }{ \sum_{j \in [d], \text{even}} \mu_j }, \quad
B_d = \frac{ \sum_{j \in [d], \text{odd}} z_j }{ \sum_{j \in [d], \text{odd}} \mu_j }, \quad\text{and}\quad
t_d = \frac{ \sum_{j \in [d], \text{even}} \mu_j }{ \mu_{1:d} }.
$$
Here, $A_d$ and $B_d$ are independent random variables, $t_d\in(0,1)$ is non-random, and $z_{1:d} / \mu_{1:d} = t_d A_d + (1-t_d) B_d$. Because $\phi(x) = (x-1)^6$ is a convex function, 
$$\Bigl(\frac{z_{1:d}}{\mu_{1:d}} - 1\Bigr)^6\le t_d (A_d-1)^6 + (1-t_d) (B_d-1)^6 \leq (A_d-1)^6 + (B_d-1)^6.$$

Next, $\theta$ is between 1 and $z_{1:d} / \mu_{1:d}$, and so $\theta^{p-6}\le1 + (z_{1:d} / \mu_{1:d})^{p-6}$. Since $t_d A_d$ and $(1-t_d) B_d$ are both lower bounds for $z_{1:d} / \mu_{1:d}$,
we can take either $1 + t_d^{p-6} A_d^{p-6}$ or $1 + (1-t_d)^{p-6} B_d^{p-6}$ as an upper bound for $\theta^{p-6}$.

Because the exponent $p-6$ is negative
we will need to 
bound $t_d$ away from zero below.  
Using upper and lower bounds on $\mu_j$
we know that
$$
t_d\ge \frac{((d/2)-1)\umu}{d\omu}.
$$
That lower bound is strictly positive for $d=3$
and it increases with $d$,
so $t_d^{p-6}=O(1)$.
A similar argument shows that $(1-t_d)^{p-6}=O(1)$ too,
and so $\max(t_d^{p-6},(1-t_d)^{p-6})\le C$
for some $C<\infty$ and all $d\ge3$.
Therefore, we find that for $d$ large enough 
\begin{align*}
&\e\Bigl[ \Bigl(\frac{z_{1:d}}{\mu_{1:d}} - 1\Bigr)^6 \theta^{p-6}\Bigr]\\
&\leq \e\big[ \big( (A_d-1)^6 + (B_d-1)^6 \big) \theta^{p-6} \big] \\
&\leq \e \big[ (A_d-1)^6 \cdot ( 1 + t_d^{p-6} B_d^{p-6} ) + (B_d-1)^6 \cdot ( 1 + (1-t_d)^{p-6} A_d^{p-6} ) \big] \\
&\leq \e[ (A_d-1)^6] \cdot (1 + C\cdot\e[B_d^{p-6}]) + \e[ (B_d-1)^6] \cdot (1 + C\cdot\e[A_d^{p-6}]). 
\end{align*}
Now both $\e((A_d-1)^6)$ and $\e((B_d-1)^6))$
are $O(d^{-3})$ by~\eqref{eq:mom6} of Proposition~\ref{prop:sixmoments}
and $\max(\e[A_d^{p-6}],\e[B_d^{p-6}])
\le \beta^{p-6} =O(1)$ by Proposition~\ref{prop:constUpperBound}.
We also note that $(z_{1:d} / \mu_{1:d} - 1)^6 \theta^{p-6}$ is non-negative, so the expectation is bounded below by zero. Therefore, 
$\e( (z_{1:d} / \mu_{1:d} - 1)^6 \theta^{p-6})=O(d^{-3})$ as required.
\end{proof}

\begin{remark}\label{rem:momentasymptote}
The implied constant in the $O(d^{-3})$ error term depends only on the constants
in bounds~\eqref{eq:boundedmean}, \eqref{eq:negmoment} and \eqref{eq:sixmomentbounds}.
\end{remark}

\section{Convergence rates for multiquadrics}\label{sec:rates}

In this section we have the main background
results to support our finding that
$\nu(f) = 1+O(1/d)$ for generalized multiquadric
RBFs under moment conditions.

For the next result we use falling factorial
notation $(p)_k=p(p-1)\cdots(p-k+1)$ where $p$
need not be an integer and $k\ge0$ is an integer.

\begin{corollary}\label{cor:varfforrate}
Suppose $p<1$ and that the first 6 central moments of the $z_j$ exist and are bounded in magnitude. The asymptotic behavior of $\var((z_{1:d} / \mu_{1:d})^p)$ is
$$ 
\frac{ p^2 \sigma_{1:d}^2 }{ (\mu_{1:d})^2 } \Bigl( 
	1
	+ (p-1) \cdot \frac{ \mu^{(3)}_{1:d} }{ \sigma^2_{1:d} \mu_{1:d} } 
	+ \frac{1}{2} (p-1)(3p-5) \cdot \frac{\sigma^2_{1:d}}{(\mu_{1:d})^2} 
	+ O(d^{-2}) \Bigr). 
$$
\end{corollary}
\begin{proof} 
Because $p<1$ we have both $p<6$
and $2p<6$. So we can use Proposition~\ref{prop:pmomentwith6}
with exponents $p$ and $2p$ to 
write $\var( (z_{1:d}/\mu_{1:d})^p)$ as
\begin{align*}
&\e\Bigl[ \Bigl( \frac{z_{1:d}}{\mu_{1:d}} \Bigr)^{2p} \Bigr] - \e\Bigl[\Bigl( \frac{z_{1:d}}{\mu_{1:d}} \Bigr)^p\Bigr]^2 \\
&= \biggl( 
	1
	+ \frac{ (2p)_2 }{2!} \cdot\frac{ \sigma_{1:d}^2 }{(\mu_{1:d})^2}
	+ \frac{ (2p)_3 }{3!} \cdot \frac{ \mu^{(3)}_{1:d} }{ (\mu_{1:d})^3}
	+ \frac{ (2p)_4 }{4!} \cdot \frac{3 (\sigma_{1:d}^2)^2 }{(\mu_{1:d})^4}
	+ O(d^{-3}) \biggr) \\
& \qquad - \biggl( 
	1
	+ \frac{ (p)_2 }{2!} \cdot \frac{ \sigma_{1:d}^2 }{(\mu_{1:d})^2}
	+ \frac{ (p)_3 }{3!} \cdot \frac{ \mu^{(3)}_{1:d} }{(\mu_{1:d})^3}
	+ \frac{ (p)_4 }{4!} \cdot \frac{ 3 (\sigma_{1:d}^2)^2 }{(\mu_{1:d})^4}
	+ O(d^{-3})
\biggr)^2 \\
&= \frac{ p^2 \sigma_{1:d}^2 }{ (\mu_{1:d})^2 } \biggl( 
	1
	+ (p-1) \cdot \frac{ \mu^{(3)}_{1:d} }{ \sigma_{1:d}^2 \mu_{1:d} } 
	+ \frac{1}{2} (p-1)(3p-5) \cdot \frac{ \sigma_{1:d}^2 }{(\mu_{1:d})^2} 
	+ O(d^{-2}) \biggr)
\end{align*}
after some algebra.
\end{proof}

\begin{proposition}\label{prop:upperboundsumtausq}
Let $z_j\ge0$ be independent random variables satisfying
the sixth moment condition~\eqref{eq:sixmomentbounds}
and the mean condition~\eqref{eq:boundedmean}.
Then for $p<1$
$$ \sum_{j=1}^d \olt_j^2 
\leq \frac{p^2 \sigma_{1:d}^2}{(\mu_{1:d})^2} \Bigl( 1 + (p-1)(2p-3) \frac{\sigma_{1:d}^2}{(\mu_{1:d})^2}
+ (p-1) \frac{ \mu^{(3)}_{1:d}}{ \mu_{1:d} \sigma_{1:d}^2 } + O(d^{-2}) \Bigr).$$
\end{proposition}
\begin{proof}
For each $j\in[d]$ we form a Taylor expansion
of $(z_{1:d}/\mu_{1:d})^p$ in powers
of $z_j-\mu_j$ as follows
\begin{align} 
&S_j^p + \frac{p}{\mu_{1:d}} S_j^{p-1} (z_j-\mu_j) + \frac{(p)_2}{2 (\mu_{1:d})^2 } S_j^{p-2} (z_j-\mu_j)^2 + \frac{ (p)_3}{6 (\mu_{1:d})^3 } (S_j')^{p-3} (z_j-\mu_j)^3\notag\\
&=: T_0+T_1+T_2+T_3\label{eq:deftk}
\end{align}
where
\begin{align*}
S_j &= \frac{ (z_{1:d} - z_j) + \mu_j }{ \mu_{1:d} }, 
\quad\text{and}\quad S_j' = \frac{ (z_{1:d} - z_j) + \theta }{ \mu_{1:d} } \notag
\end{align*}
for some $\theta$ between $\mu_j$ and $z_j$.

Now $\ult^2_j = \e(\var( (z_{1:d}/\mu_{1:d})^p\giv\bsx_{-j})$
so we begin by bounding the conditional variances of 
the terms $T_k$ defined at equation~\eqref{eq:deftk}.
Because $S_j$ is a function of $\bsz_{-j}$,
$\var(T_0\giv\bsx_{-j})=\var(S_j\giv\bsz_{-j})=0$.
Similarly
$$
\var(T_1\giv\bsz_{-j}) = \frac{p^2}{(\mu_{1:d})^2}S_j^{2p-2}\sigma^2_j.
$$
Next, noting that $\var( (z_j-\mu_j)^2\giv\bsz_{-j})
=\var( (z_j-\mu_j)^2)\le \mu^{(4)}_j\le\lambda$,
\begin{align*}
\var(T_2\giv\bsx_{-j})
&=\frac{(p)_2^2}{4(\mu_{1:d})^4}
S_j^{2p-2}
\var((z_j-\mu_j)^2\giv\bsz_{-j})
\le\frac{(p)_2^2}{4(\mu_{1:d})^4}S_j^{2p-2}\lambda.
\end{align*}
Turning to the one term with $S'_j$
\begin{align*}
\var( T_3\giv\bsz_{-j}) &\le
\frac{(p)_3^2}{36(\mu_{1:d})^6}
\e( (S'_j)^{2p-6}(z_j-\mu_j)^6\giv\bsz_{-j})\\
&\le 
\frac{(p)_3^2}{36(\mu_{1:d})^6}
\e\Bigl( 
\Bigl(\frac{z_{1:d}-z_j}{\mu_{1:d}}\Bigr)^{2p-6}
(z_j-\mu_j)^6\giv\bsz_{-j}\Bigr)\\
&\le 
\frac{(p)_3^2}{36(\mu_{1:d})^6}
\Bigl(\frac{z_{1:d}-z_j}{\mu_{1:d}}\Bigr)^{2p-6}
\lambda.
\end{align*}

With the above decomposition, we write
\begin{align*}
\olt^2_j&\le\e( \var(T_1\giv\bsz_{-j}))
+\e( \var(T_2\giv\bsz_{-j}))
+\e( \var(T_3\giv\bsz_{-j}))\\
&+2\e( \cov(T_1,T_2\giv\bsz_{-j}))
+2\e( \cov(T_1,T_3\giv\bsz_{-j}))
+2\e( \cov(T_2,T_3\giv\bsz_{-j})).
\end{align*}

Proposition~\ref{prop:pmomentwith6} shows that
for $q<0$ and large enough $d$
\begin{align*}
\e[ S_j^q ]=
\e\Bigl[ \Bigl(\frac{z_{1:d}-z_j + \mu_j }{ \mu_{1:d}}\Bigr)^q\Bigr]&=1+\frac{q(q-1)}2
\frac{\sigma^2_{1:d}-\sigma^2_j}{(\mu_{1:d})^2}
+O\Bigl(\frac1{d^2}\Bigr)\\
&=1+\frac{q(q-1)}2
\frac{\sigma^2_{1:d}}{(\mu_{1:d})^2}
+O\Bigl(\frac1{d^2}\Bigr).
\end{align*}
In this application of Proposition~\ref{prop:pmomentwith6}, the variable $z_j$ with variance $\sigma^2_j$ is replaced by
$\mu_j$ with variance $0$. That proposition does
not assume strictly positive $\sigma^2_j$.
Note that the implied constant within  $O(1/d^2)$ depends only on moment conditions 
from Remark~\ref{rem:momentasymptote}.
and can be bounded independently of $j$. 

Also
\begin{align*}
0\le\e\Bigl[S_j^{p-1}
\Bigl(\frac{z_{1:d}-z_j}{\mu_{1:d}}\Bigr)^{p-3}
\Bigr]
\le\e[ S_j^{2p-4}]\Bigl(1+\frac{\omu}{(d-1)\umu}\Bigr)^{3-p}
=1+O\Bigl(\frac1d\Bigr).
\end{align*}
Then the expected variances are
\begin{align*}
\e(\var(T_1\giv\bsz_{-j}))&=
\Bigl(\frac{p}{\mu_{1:d}}\Bigr)^2\sigma^2_j
\e( S_j^{2p-2})\\
&\le\Bigl(\frac{p}{\mu_{1:d}}\Bigr)^2\sigma^2_j
\Bigl(
1+\frac{(2p-2)(2p-3)}2
\frac{\sigma^2_{1:d}}{(\mu_{1:d})^2}+O\Bigl(\frac1{d^2}\Bigr)
\Bigr), \\
\e(\var(T_2\giv\bsz_{-j}))&\le
\Bigl(\frac{(p)_2}{2(\mu_{1:d})^2}\Bigr)^2
\e( S_j^{2p-2})\lambda =O\Bigl(\frac1{d^4}\Bigr),\quad\text{and}\\
\e(\var(T_3\giv\bsz_{-j})) &
\le\frac{(p)_3^2}{36(\mu_{1:d})^6}
\e\Bigl(\Bigl(\frac{z_{1:d}-z_j}{\mu_{1:d}-\mu_j}\Bigr)^{2p-6}\Bigr)
\lambda=O\Bigl(\frac1{d^6}\Bigr).
\end{align*}
Because $S_j$ is a function of $\bsz_{-j}$,
\begin{align*}
\cov(T_1,T_2\giv\bsz_{-j})
&= \frac{p(p)_2}{2(\mu_{1:d})^3}S_j^{2p-3}\cov(z_j-\mu_j,
(z_j-\mu_j)^2)
\\&
=
\frac{p(p)_2}{2(\mu_{1:d})^3}S_j^{2p-3}\mu^{(3)}_j,
\end{align*}
and so
\begin{align*}
\e(\cov(T_1,T_2\giv\bsz_{-j}))
&=\frac{p(p)_2}{2(\mu_{1:d})^3}\e(S_j^{2p-3})\mu^{(3)}_j\\
&=
\frac{p(p)_2}{2(\mu_{1:d})^3}
\Bigl(1 + O\Bigl(\frac1{d}\Bigr)\Bigr)\mu^{(3)}_j.
\end{align*}
Similarly
\begin{align*}
\cov(T_1,T_3\giv\bsz_{-j})
&\le \frac{p(p)_3}{6(\mu_{1:4})^4}S_j^{p-1}
\Bigl(\frac{z_{1:d}-z_j}{\mu_{1:d}}\Bigr)^{p-3}\lambda,
\quad\text{and}\\
\cov(T_2,T_3\giv\bsz_{-j})
&\le \frac{p(p)_2}{12(\mu_{1:4})^5}S_j^{p-1}
\Bigl(\frac{z_{1:d}-z_j}{\mu_{1:d}}\Bigr)^{p-3}\lambda.
\end{align*}
so that
\begin{align*}
\e[\cov(T_1,T_3\giv\bsz_{-j})]
=O\Bigl(\frac1{d^4}\Bigr)\quad\text{and}\quad
\e[\cov(T_2,T_3\giv\bsz_{-j})]
=O\Bigl(\frac1{d^5}\Bigr).
\end{align*}

Combining all of our bounds
\begin{align*}
\olt_j^2
&\le 
\Bigl(\frac{p}{\mu_{1:d}}\Bigr)^2\sigma^2_j
\Bigl(
1+(2p-1)(p-3)\frac{\sigma^2_{1:d}}{(\mu_{1:d})^2}
+O\Bigl(\frac1{d^2}\Bigr)\Bigr)\\
&\phe+\frac{p(p)_2}{2(\mu_{1:d})^3}\mu_j^{(3)}
+O\Bigl(\frac1{d^4}\Bigr).
\end{align*}
The implied constants in both $O(\cdot)$
expressions above can be chosen
independently of $j$ from Remark~\ref{rem:momentasymptote}.
Then summing over $j$ yields
$$
\sum_{j=1}^d\olt^2_j
\le
\frac{p^2\sigma^2_{1:d}}{(\mu_{1:d})^2}
\Bigl(
1+(2p-1)(p-3)\frac{\sigma^2_{1:d}}{(\mu_{1:d})^2}
+\frac{p-1}2\frac{\mu_{1:d}^{(3)}}{\mu_{1:d}\sigma^2_{1:d}}
\Bigr)+O\Bigl(\frac1{d^{3}}\Bigr).
$$
\end{proof}

\section{Mean dimension
approaching one}\label{sec:boundednu}

Here we prove the Lemmas needed for
Theorem~\ref{thm:limisone}.
We have a subsection to prove upper
bounds on Sobol' indices and another
for lower bounds on the variance.

\subsection{Sobol' index upper bounds}\label{sec:sobolupper}
Here we find upper bounds for the Sobol' indices 
$\olt^2_j$ that form the numerator of $\nu(f)$.
We will need some properties of
\begin{align}\label{eq:deftd}
T_d := \left( \frac{z_{1:d} - z_J }{\mu_{1:d}} \right)^p,
\end{align}
where $J\in[d]$ is a random index with
\begin{align}\label{eq:rindex}
\Pr(J=j) = \frac{\sigma_j^2}{\sigma_{1:d}^2},\quad 1\le j\le d,
\end{align}
chosen independently of $\bsz$.
In particular, we need to show that $\e( |T_d-1|)\to0$
as $d\to\infty$.

\begin{proposition}\label{prop:boundSJmoments}
Let independent $z_j\ge0$ 
satisfy the lower bound condition~\eqref{eq:boundedmean}
for some $\underline{\mu}>0$,
the variance bounds~\eqref{eq:boundedvar}
and the negative moment condition~\eqref{eq:negmoment} for
some $\alpha>0$ and $M_\alpha<\infty$ for all $j=1,\dots,d$.
If the random index $J$ 
satisfies~\eqref{eq:rindex} and is independent of $\bsz$, 
then for $d>1$
$$\e\biggl( \Bigl( \frac{z_{1:d} - z_J}{\mu_{1:d}}\Bigr)^{-\alpha (d-1)} \biggr) 
\leq \beta^{-\alpha (d-1) } e^{\alpha}.
$$
where $\beta$ is given at equation~\eqref{eq:quantitybeta}.
\end{proposition}
\begin{proof}
Directly, we find that:
\begin{align*}
\e\left[ \left( \frac{z_{1:d} - z_J}{\mu_{1:d}} \right)^{-\alpha(d-1)} \right]
&= \e\left[ \left( \frac{z_{1:d} - z_J }{d-1} \right)^{-\alpha(d-1)}  \right] \cdot \left( \frac{ d-1 }{ \mu_{1:d} } \right)^{-\alpha (d-1) } \\
&\leq \e\left[ \prod_{j \in [d] \setminus\{J\}} z_j^{-\alpha} \right] \cdot  \left( \frac{ d-1 }{ d \underline \mu  } \right)^{-\alpha (d-1) } \\
&\leq (M_\alpha)^{d-1} \cdot \underline \mu^{\alpha (d-1) }
\Bigl(\frac{d}{d-1}\Bigr)^{\alpha(d-1)}
\\%e^\alpha \\
&\leq \beta^{ - \alpha (d-1) }e^\alpha.\qedhere
\end{align*}
\end{proof}

\begin{lemma}\label{lem:tdto1}
Let independent random variables 
$z_j\ge0$ satisfy the upper and lower bound mean conditions
in~\eqref{eq:boundedmean}
and the variance bounds in~\eqref{eq:boundedvar}.
Let the index $J$ be chosen
according to equation~\eqref{eq:rindex} independently of $\bsz$.
If $T_d$ is defined by equation~\eqref{eq:deftd}
with $p<0$ then $\e( |T_d-1|)\to0$ as $d\to\infty$.
\end{lemma}

\begin{proof}
We will show that $T_d$ converges to $1$ in probability and
that $T_d$ is uniformly integrable for large enough $d$. Then the result follows
by the Vitali convergence theorem.

Writing
$$
\frac{z_{1:d}-z_J}{\mu_{1:d}}
=\frac{z_{1:d}}{\mu_{1:d}}-\frac{z_J}{\mu_{1:d}}
$$
we see that the first term converges to one in probability (by our variance assumptions)
and the second term
converges to zero in probability
by our assumptions on $\mu_j$.
Therefore $(z_{1:d}-z_J)/\mu_{1:d}$ converges to one in probability
and then, by continuity $T_d$
converges to one in probability
as $d\to\infty$.

Now, we prove that $T_d$ is uniformly integrable for all $d >\lceil 1 - 2 p/ \alpha \rceil$, so that $1 + \alpha(d-1)/ p < -1$. Consider any $\epsilon > 0$, and select any value $M > \min( \beta^p, \beta^{2p} e^\alpha / \epsilon )$. Noting that $x \mapsto x^{-\alpha (d-1)/p}$ is a monotonically increasing function and then using Proposition~\ref{prop:boundSJmoments},

\begin{align*}
\int_M^\infty \Pr(|T_d| \geq z) \rd z 
&= \int_M^\infty \Pr\left( \left( \frac{ z_{1:d} - z_J }{\mu_{1:d}} \right)^p \geq z \right) \rd z \\
&= \int_M^\infty \Pr\biggl( \left( \frac{ z_{1:d} - z_J }{\mu_{1:d}} \right)^{-\alpha (d-1)} \geq z^{-\alpha (d-1)/p} \biggr) \rd z \\
&\leq \int_M^\infty \e\biggl[ \left( \frac{ z_{1:d} - z_J }{\mu_{1:d}} \right)^{-\alpha (d-1)} \biggr] z^{\alpha (d-1)/p} \rd z \\
&\leq \beta^{-\alpha(d-1)} e^\alpha \int_M^\infty z^{\alpha (d-1)/p} \rd z \\
&= \beta^{-\alpha(d-1)} e^\alpha 
\cdot \frac{ M^{1+\alpha(d-1)/p} }{ - (1 + \alpha (d-1) / p)}\\
&\leq \beta^p e^\alpha \cdot (M \beta^{-p})^{1+\alpha(d-1)/p}\\
&\leq \beta^p e^{\alpha} (M\beta^{-p})^{-1}  \tag{as $M\ \beta^{-p} \geq 1$ and $1+\alpha(d-1)/p \leq -1$} \\
&= \beta^{2p} e^{\alpha} / M \\
&\leq \epsilon
\end{align*}
because $M \geq \beta^{2p} e^\alpha / \epsilon$.

Therefore
$\int_M^\infty \Pr(T_d| \geq z) \rd z 
\leq  \epsilon$.
It follows that $T_d$ is uniformly integrable for all $d\geq \lceil 1 - 2p / \alpha \rceil$, which completes our claim.
\end{proof}

\begin{proposition}\label{prop:condlipschitz}
Let $h:\real^2\to\real$ be a function 
where $h(\cdot,z)$ is an $M(z)$-Lipschitz function
for every $z$.
If $x$ and $z$ are independent random variables, then
$$ \e[ \var( h(x,z) \giv z) ] \leq \e[ M(z)^2] \cdot \var(x).$$
\end{proposition}
\begin{proof}
First $\var( h(x,z)\giv z)\le M(z)^2\var(x\giv z)=M(z)^2\var(x)$ 
by independence of $x$ and $z$.  The result follows by
taking the expectation over $z$.
\end{proof}

\begin{lemma}\label{lem:sobindexupper}
Let independent $z_j\ge0$ satisfy 
the upper and lower mean bounds in \eqref{eq:boundedmean}
and the upper and lower variance bounds in \eqref{eq:boundedvar}.
Let $f(\bsz) = (z_{1:d}/\mu_{1:d})^p$ for non-zero $p<1$.
Then
$$ 
\limsup_{d \rightarrow \infty} \frac{ \sum_{j=1}^d \olt_j^2 }{ p^2 \cdot \frac{\sigma_{1:d}^2 }{ (\mu_{1:d})^2 } }
\leq 1.
$$
\end{lemma}
\begin{proof} 
For all $j\in[d]$,
$$0\le
\frac{\partial}{\partial z_j}f(\bsz) =
\frac{p}{\mu_{1:d}}\Bigl(\frac{z_{1:d}}{\mu_{1:d}}\Bigr)^{p-1}
\le\frac{p}{\mu_{1:d}}\Bigl(\frac{z_{1:d}-z_j}{\mu_{1:d}}\Bigr)^{p-1}
$$
which we can use as a conditional Lipschitz bound
independent of $z_j$.
Then using the identity~\eqref{eq:econdvar} and Proposition~\ref{prop:condlipschitz}
\begin{align*}
\olt^2_j &= \e\Bigl(\var\Bigl( \Bigl(\frac{z_{1:d}}{\mu_{1:d}}\Bigr)^p\giv z_{-j}\Bigr)\Bigr)
 \le\frac{\sigma^2_jp^2}{(\mu_{1:d})^2}
\e\Bigl(\Bigl(\frac{z_{1:d}-z_j}{\mu_{1:d}}\Bigr)^{2p-2}\Bigr).
\end{align*}

Now
\begin{align*}
%\limsup_{d\to\infty}
\frac1{p^2\frac{\sigma^2_{1:d}}
{(\mu_{1:d})^2}}
\sum_{j=1}^d\olt^2_j
&\le
%\limsup_{d\to\infty}
\sum_{j=1}^d\frac{\sigma^2_j}{\sigma^2_{1:d}}
\e\Bigl(\Bigl(\frac{z_{1:d}-z_j}{\mu_{1:d}}\Bigr)^{2p-2}\Bigr)
\end{align*}
which we recognize as $\e(T_d)$ defining $T_d$
as at~\eqref{eq:deftd} but with exponent $2p-2<0$.
Then Lemma~\ref{lem:tdto1} finishes the proof.
\end{proof}

\subsection{Variance lower bounds}\label{sec:varlower}

In Section~\ref{sec:sobolupper} we found an upper bound
for a normalized upper bound of Sobol' indices.
Here we get a lower bound for the variance
of the radial basis functions.

We will use the following inequality.
If $Y_d$ for $d\ge1$ are random variables
that have finite variances and converge
in distribution to a random variable $Y$,
then
\begin{align}\label{eq:liminflowerbound}
\liminf_{d \rightarrow \infty} \var(Y_d) \geq \var(Y).
\end{align}
\begin{lemma}\label{lem:varlowerbd}
Let independent $z_j\ge0$ satisfy 
the upper and lower mean bounds in \eqref{eq:boundedmean}
and the upper and lower variance bounds in \eqref{eq:boundedvar}.
Let $f(\bsz) = (z_{1:d}/\mu_{1:d})^p$ for non-zero $p<1$.
Then
$$ \liminf_{d \rightarrow \infty} \var(f(\bsz)) \cdot \left( p^2 \cdot \frac{\sigma_{1:d}^2}{ (\mu_{1:d})^2 } \right)^{-1} \geq 1.$$
\end{lemma}
\begin{proof} 
From the mean value theorem
\begin{align*} 
\frac{\mu_{1:d}}{\sigma^2_{1:d}} (f(\bsz) - 1)
&= \frac{\mu_{1:d}}{\sigma_{1:d}^2} \left[ \left( \frac{z_{1:d}}{\mu_{1:d}} \right)^p - 1 \right] 
= p \theta^{p-1} \cdot \frac{z_{1:d}-\mu_{1:d}}{\sigma^2_{1:d}} 
\end{align*}
for some $\theta$ between 1 and 
$z_{1:d}/\mu_{1:d}$. That ratio converges to 1 in probability
as $d\to\infty$ and so $\theta\to1$ in probability.
Then by the continuous mapping theorem,
$p \theta^{p-1}$ converges to $p$ in probability.

Next  $(z_{1:d}-\mu_{1:d})/\sigma^2_{1:d}
\tod\dnorm(0,1)$ by the central limit theorem
and so using Slutsky's theorem
$$\frac{\mu_{1:d}}{\sigma^2_{1:d}}( f(\bsz)-1)
\tod\dnorm(0,p^2).$$
Finally, from equation~\eqref{eq:liminflowerbound}
\begin{align*}
\liminf_{d \rightarrow \infty} \frac{ \var(f(\bsz)) }{ p^2 \cdot \frac{\sigma_{1:d}^2}{(\mu_{1:d})^2 } } 
&=
\liminf_{d \rightarrow \infty} \frac{1}{p^2} \var\left( \frac{\mu_{1:d}}{\sigma_{1:d}} [f(z_{1:d})-1] \right)\\
&\geq
\frac{1}{p^2} \cdot \var( \dnorm(0,p))
=1.\qedhere\end{align*}
\end{proof}
\end{document}